\documentclass{amsart}
\usepackage{graphicx}
\usepackage{amsmath}
\usepackage{amsthm}
\usepackage{amsfonts,enumerate}
\usepackage{tikz,tkz-graph}

\newcommand{\NN}{{\mathbb N}}
\newcommand{\ZZ}{{\mathbb Z}}
\newcommand{\PP}{{\mathbb P}}

\newtheorem{theorem}{Theorem}
\newtheorem{lemma}[theorem]{Lemma}
\newtheorem{corollary}[theorem]{Corollary}
\newtheorem{proposition}[theorem]{Proposition}
\newtheorem{conjecture}[theorem]{Conjecture}
\newtheorem{notation}[theorem]{Notation}
\newtheorem{remark}[theorem]{Remark}
\theoremstyle{definition}
\newtheorem{example}{Example}
\begin{document}

\title[On the smallest trees with the same restricted $U$-polynomial]{On the smallest trees with the same restricted $U$-polynomial and the rooted $U$-polynomial}

\author{Jos\'e Aliste-Prieto \and Anna de Mier \and Jos\'e Zamora}
\address{Jos\'e Aliste-Prieto. Departamento de Matematicas, Universidad Andres Bello, Republica 498, Santiago, Chile}
\email{jose.aliste@unab.cl}

\address{Jos\'e Zamora. Departamento de Matem\'aticas, Universidad Andres Bello, Republica 498, Santiago, Chile
}
\email{anna.de.mier@upc.edu}

\address{Anna de Mier. Departament de Matem\`atiques, Universitat Polit\`ecnica de Catalunya, Jordi Girona 1-3, 08034 Barcelona, Spain
}
\email{josezamora@unab.cl}
\maketitle

\begin{abstract}
In this article, we construct explicit examples of pairs of non-isomorphic trees with the same restricted $U$-polynomial for every $k$; by this we mean that the polynomials agree on terms with degree at most $k+1$. The main tool for this construction is a generalization of the $U$-polynomial to rooted graphs, which we introduce and study in this article. Most notably we show that rooted trees can be reconstructed from its rooted $U$-polynomial. 
\end{abstract}
%

\tikzstyle{every node}=[circle, draw, fill=black,
                        inner sep=0pt, minimum width=4pt,font=\small]



\section{Introduction}\label{sec:intro}

The chromatic symmetric function \cite{stanley95symmetric} and the $U$-polynomial \cite{noble99weighted} are powerful graph invariants as they generalize many other invariants like,  for instance, the chromatic polynomial, the matching polynomial and the Tutte polynomial. It is well known that the chromatic symmetric function  and the $U$-polynomial are equivalent when restricted to trees, and there are examples of non-isomorphic graphs with cycles having the same $U$-polynomial (see \cite{brylawski1981intersection} for examples of graphs with the same polychromate and \cite{sarmiento2000polychromate,merino2009equivalence} for the equivalence between the polychromate and the $U$-polynomial) and also the same is true for the chromatic symmetric function (see \cite{stanley95symmetric}) .  However, it is an open question to know whether there exist non-isomorphic trees with the same chromatic symmetric function (or, equivalently, the same $U$-polynomial). The negative answer to the latter question, that is, the assertion that two trees that have the same chromatic symmetric function must be isomorphic, is sometimes referred to in the literature as \emph{Stanley's (tree isomorphism) conjecture}. This conjecture has been so far verified for trees up to 29 vertices~\cite{heil2018algorithm} and also for some classes of trees, most notably caterpillars \cite{aliste2014proper,loebl2019isomorphism} and spiders \cite{martin2008distinguishing}.

A natural simplification for Stanley's conjecture is to define a truncation of the $U$-polynomial, and then search for non-isomorphic trees with the same truncation. A study of these examples  could help to better understand the picture for solving Stanley's conjecture. 
To be more precise, suppose that $T$ is a tree with $n$ vertices. Recall that a parition $\lambda$ of $n$ is a sequence $\lambda_1,\lambda_2,\ldots,\lambda_l$ where $\lambda_1\geq \lambda_2\geq\cdots\geq \lambda_l$.Recall that $U(T)$ can be expanded  as
\begin{equation}
\label{eq:Uintro}
U(T)=\sum_{\lambda} c_\lambda \mathbf x_\lambda,
\end{equation}
 where the sum is over all partitions $\lambda$ of $n$, $\mathbf x_\lambda=x_{\lambda_1}x_{\lambda_2}\cdots x_{\lambda_l}$ and the $c_\lambda$ are non-negative integer coefficients (for details of this expansion see Section 2). 
In a previous work~\cite{Aliste2017PTE}, the authors studied the $U_k$-polynomial  defined by restricting {the sum in \eqref{eq:Uintro} to the partitions of length smaller or equal than $k+1$}, and then showed the existence of non-isomorphic trees with the same $U_k$-polynomial for every $k$. This result is based on a remarkable connection between the $U$-polynomial of a special class of trees and the Prouhet-Tarry-Escott problem in number theory. Although the Prouhet-Tarry-Escott problem is known to have solutions for every $k$, in general it is difficult to find explicit solutions, specially if $k$ is large. Hence, it was difficult to use this result to find explicit examples of trees with the same $U_k$-polynomial.

The main result of this paper is to give an explicit and simple construction of non-isomorphic trees with the same $U_k$-polynomial for every $k$. It turns out that for $k=2,3,4$ our examples coincide with the minimal examples already found by Smith, Smith and Tian~\cite{smith2015symmetric}.  This leads us to conjecture that for every $k$ our construction yields the smallest non-isomorphic trees with the same $U_k$-polynomial. We also observe that if this conjecture is true, then Stanley's conjecture is also true.

To prove our main result, we first introduce and study a generalization of the $U$-polynomial to rooted graphs, which we call the rooted $U$-polynomial or $U^r$-polynomial. As it is the case for several invariants of rooted graphs,  the rooted $U$-polynomial distinguishes rooted trees up to isomorphism. Under the correct interpretation, it can also be seen as a generalization of the pointed chromatic symmetric function introduced in \cite{pawlowski2018chromatic} (See Remark \ref{pawlowski}).  The key fact for us is that  the rooted $U$-polynomial exhibits simple product formulas when applied to some joinings of rooted graphs. These formulas together with some non-commutativity is what allows our constructions to work. 

Very recently, another natural truncation for the $U$-polynomial was considered in \cite{heil2018algorithm}. Here, they restrict the range of the sum in \eqref{eq:Uintro} to partitions whose parts are smaller or equal than $k$. They also verified that trees up to $29$ vertices are distinguished by the truncation with $k=3$ and proposed the conjecture that actually $k=3$ suffices to distinguish all trees. 

This paper is organized as follows. In Section \ref{sec:rooted}, we introduce the rooted $U$-polynomial and prove our main product formulas. In Section \ref{sec:dist}, we show that the rooted $U$-polynomial distinguishes rooted trees up to isomorphism. In Section \ref{sec:main}, we recall the definition of the $U_k$-polynomial and prove our main result.

\section{The rooted $U$-polynomial}\label{sec:rooted}

We give the definition of the $U$-polynomial first introduced by 
Noble and Welsh~\cite{noble99weighted}. 
We consider graphs where we allow  loops and parallel edges.

Let $G = (V, E)$ be a graph. Given $A\subseteq E$, the restriction $G|_A$ of $G$ to $A$ is the subgraph of $G$ obtained from $G$ after deleting every edge that is not contained in $A$ (but keeping all the vertices). The \emph{rank} of $A$ is defined as $r(A) = |V| - k(G|_A)$, where $k(G|_A)$ is the number of connected components of $G|_A$.
The \emph{partition induced by $A$}, denoted by $\lambda(A)$, is the partition of $|V|$ whose parts are the sizes of the  connected components of $G|_A$.

Let $y$ be an indeterminate and $\mathbf{x} = x_1,x_2,\ldots$ be an infinite set of commuting indeterminates that commute with $y$. Given an integer partition $\lambda=\lambda_1,\lambda_2,\cdots,\lambda_l$, define  $\mathbf{x}_\lambda:=x_{\lambda_1}\cdots x_{\lambda_l}$. The \emph{$U$-polynomial} of a graph $G$ is defined as 
\begin{equation}
\label{def:W_poly}
U(G;\mathbf x, y)=\sum_{A\subseteq E}\mathbf x_{\lambda(A)}(y-1)^{|A|-r(A)}.
\end{equation}

Now we recall the definition of the $W$-polynomial for weighted graphs, from which the $U$-polynomial is a specialization. A \emph{weighted graph} is a pair $(G,\omega)$ where $G$ is a graph and $\omega:V(G)\rightarrow \PP$ is a function. We say that $\omega(v)$ is the weight of the vertex $v$.
Given a weighted graph $(G,\omega)$ and an edge $e$, 
the  graph $(G-e,\omega)$ is defined by deleting the edge $e$ and leaving $\omega$ unchanged. If $e$ is not a loop, then the graph $(G/e,\omega)$ is defined by first deleting $e$ then by identifying the vertices $u$ and $u'$ incident to $e$ into a new vertex $v$. We set $\omega(v)=\omega(u)+\omega(u')$ and leave all other weights unchanged.

The $W$-polynomial of a weighted graph $(G,\omega)$ is defined by the following properties:  
\begin{enumerate}
\item If $e$ is not a loop, then $W(G,\omega)$ satisfies
\[W(G,\omega) = W(G-e,\omega) + W(G/e,\omega);\]
\item If $e$ is a loop, then 
\[W(G,\omega) = y W(G-e,\omega);\]
\item If $G$ consists only of isolated vertices $v_1,\ldots,v_n$ with weights $\omega_1,\ldots,\omega_n$, then 
\[W(G,\omega) = x_{\omega_1}\cdots x_{\omega_n}.\]
\end{enumerate}
In \cite{noble99weighted}, it is proven that the $W$-polynomial is well-defined and that $U(G)=W(G,1_G)$ where $1_G$ is the weight function assigning weight $1$ to all vertices of $G$. The deletion-contraction formula is very powerful, but in this paper we will only use it in the beginning of the proof of Theorem \ref{theo:YZ} in Section \ref{sec:main}.

A \emph{rooted graph} is a pair $(G,v_0)$, where $G$ is a graph and $v_0$ is a vertex of $G$ that we call the \emph{root} of $G$. 
Given $A\subseteq E$, define $\lambda_r(A)$ to be the size of the component of $G|_A$ that contains the root $v_0$, and $\lambda_-(A)$ to be 
the partition induced by the sizes of all the other components. The \emph{rooted $U$-polynomial} is  
\begin{equation}
\label{def:U_poly_rooted}
U^r({G,v_0};\mathbf x, y, z)=\sum_{A\subseteq E}\mathbf x_{\lambda_{-}(A)}z^{\lambda_{r}(A)}(y-1)^{|A|-r(A)},
\end{equation}
where $z$ is a new indeterminate that commutes with $y$ and $x_1,x_2,\ldots$.
We often write $G$ instead of $(G,v_0)$ when $v_0$ is clear from the context,  and so we will write 
$U^r(G)$ instead of $U^r(G,v_0)$. Also, if $(G,v_0)$ is a rooted graph, we will write $U(G)$ for the $U$-polynomial of $G$ (seen as an unrooted graph). If we compare $U^r(G)$ with $U(G)$, then we see that for each term of the form $\mathbf{x}_\lambda y^n z^m$ appearing in   $U^r(G)$ there is a corresponding term of the form $\mathbf{x}_\lambda y^n x_m$  in $U(G)$. This motivates the following notation and lemma, whose proof follows directly from the latter observation.

\begin{notation}
If $P(\mathbf{x}, y, z)$ is a polynomial, then $(P(\mathbf{x},y,z))^*$ is the polynomial obtained by expanding $P$ as a polynomial in $z$ (with coefficients that are polynomials in $\mathbf{x}$ and $y$) and then substituting $z^n\mapsto x_n$ for every $n\in\NN$. For instance, 
if $P(\mathbf{x},y,z) = x_1yz-x_2x_3z^3$, then $P(\mathbf{x},y,z)^* = x_1^2y-x_2x_3^2$. Note that in general 
$(P(\mathbf{x},y,z)Q(\mathbf{x},y,z))^* \neq P(\mathbf{x},y,z)^*Q(\mathbf{x},y,z)^*$.
\end{notation}

\begin{lemma}
For every graph $G$ we have
\begin{equation}
(U^r(G))^* = U(G).
\end{equation}
\end{lemma}
\begin{remark}
We could also define a rooted version of the $W$-polynomial, but we will not need this degree of generality for the purposes of this article. 
\end{remark}



\subsection{Joining of rooted graphs and product formulas}
In this section we show two product formulas for the rooted $U$-polynomial. These will play a central role in the proofs of the results in the following sections.
Let $(G,v)$ and $(H,v')$ be two  rooted graphs. Define
$G\odot H$ to be the rooted graph obtained after first taking the disjoint union of $G$ and $H$ and then by identifying $v$ and $v'$. We refer to $G\odot H$ as \emph{the joining} of $G$ and $H$. Note that from the definition it is clear that $G\odot H = H\odot G$. 
We also define $G\cdot H$ to be the rooted graph obtained after first taking the disjoint union of $G$ and $H$, then adding an edge between $v$ and $v'$ and finally declaring $v$ as the root of the resulting graph. Since we made a choice for the root, in general $G\cdot H$ and $H\cdot G$ are isomorphic as unrooted graphs, but not as rooted graphs. 
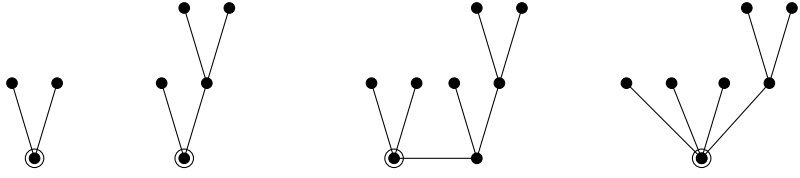
\begin{figure}
\begin{center}
\begin{tikzpicture}
\draw (0.5,0) node[fill=none,minimum width=7pt] {};
\draw (0.2,1) node {} -- (0.5,0)
            node {} -- (0.8,1)
            node {};
\end{tikzpicture}
\hspace{1cm}
\begin{tikzpicture}
\draw (0.5,0) node[fill=none,minimum width=7pt] {};
\draw (0.2,1) node {} -- (0.5,0)
            node {} -- (0.8,1)
            node {} -- (1.1,2)
            node {};
\draw (0.8,1)  -- (0.5,2)
            node {};
\end{tikzpicture}
\hspace{1.5cm}
\begin{tikzpicture}
\draw (0.5,0) node[fill=none,minimum width=7pt] {};
\draw (0.2,1) node {} -- (0.5,0)
            node {} -- (0.8,1)
            node {};
\draw (0.5,0)-- (1.6,0);
\draw (1.3,1) node {} -- (1.6,0)
            node {} -- (1.9,1)
            node {} -- (1.6,2)
            node {};
\draw (1.9,1)  -- (2.2,2)
            node {};
\end{tikzpicture}
\hspace{1cm}
\begin{tikzpicture}
\draw (0.5,0) node[fill=none,minimum width=7pt] {};
\draw (-0.5,1) node {} -- (0.5,0)
            node {} -- (0.1,1)
            node {};
\draw (0.8,1) node {} -- (0.5,0)
            node {} -- (1.4,1)
            node {} -- (1.1,2)
            node {};
\draw (1.4,1)  -- (1.7,2)
            node {};
\end{tikzpicture}
\end{center}
\caption{Example of two rooted graphs $G$ and $H$  and their different products $G\cdot H$ and $G\odot H$.} 
\end{figure}
\begin{lemma}
\label{lemma:joining}
Let $G$ and $H$ be two rooted graphs. We have
\begin{equation}
\label{eq:pseudo}
U^r(G\odot H) = \frac{1}{z}U^r(G)U^r(H).
\end{equation}
\end{lemma}
\begin{proof}
By substituting the definition of $U^r$ to $G$ and $H$ in the r.h.s. of \eqref{eq:pseudo} 
\begin{equation}
\label{W_poly_rooted}
\sum_{A_G\subseteq E(G)}\sum_{A_H\subseteq E(H)}\mathbf x_{\lambda_{-}(A_G)\cup\lambda_{-}(A_H)}z^{\lambda_{r}(A_G)+\lambda_{r}(A_H)-1}(y-1)^{|A_G|+|A_H|-r(A_G)-r(A_H)}.
\end{equation}
Given $A_G\subseteq E(G)$ and $A_H\subseteq E(H)$, set $A = A_G\cup A_H$. By the definition of the joining, there is a set $A'\subseteq E(G\odot H)$ corresponding to $A$ such that  $\lambda_{-}(A') = \lambda_{-}(A_G)\cup\lambda_{-}(A_H)$ and $\lambda_r(A) = \lambda_r(A_G)+\lambda_r(A_H)-1$. From these equations, one checks that $r(A)=r(A_G)+r(A_H)$. Plugging  these relations into \eqref{W_poly_rooted} and then rearranging the sum yields $U^r(G\odot H)$ and the conclusion now follows. 
\end{proof}
\begin{lemma}
Let $G$ and $H$ be two rooted graphs. Then we have 
\begin{equation}
\label{eq:sep_concat}
 U^r(G\cdot H) = U^r({G})(U^r(H) + U(H)).
\end{equation}
\end{lemma}
\begin{proof}
By definition, $E(G\cdot H) = E(G)\cup E(H)\cup\{e\}$, where $e$ is the edge joining the roots of $G$ and $H$. Thus, given $A\subseteq E(G\cdot H)$, we can write it as $A = A_G\cup A_H\cup F$ where $A_G\subseteq E(G)$, $A_H\subseteq E(H)$ and $F$
is either empty or $\{e\}$. Let $\delta_F$ equal to one if $F=\{e\}$ and zero otherwise. The following relations are easy to check: 
\begin{eqnarray*}
\lambda_{-}(A)&=&\begin{cases}
\lambda_{-}(A_G)\cup \lambda(A_H), &\text{if $F=\emptyset$,}\\
\lambda_{-}(A_G)\cup \lambda_{-}(A_H),&\text{otherwise};
\end{cases}\\
\lambda_r(A)&=& \lambda_r(A_G)+\lambda_r(A_H)\delta_{F};\\
r(A)&=& r(A_G)+r(A_H)+\delta_{F};\\
|A|&=&|A_G|+|A_H|+\delta_F. 
\end{eqnarray*}
Now replacing the expansions of $U^r(G)$, $U^r(H)$ and $U(H)$ into the r.h.s. of 
\eqref{eq:sep_concat} yields 
\begin{multline}
\sum_{A_G\subseteq E(G),A_H\subseteq E(H)} \mathbf x_{\lambda_{-}(A_G)\cup\lambda_{-}(A_H)} z^{\lambda_{r}(A_G)+\lambda_{r}(A_H)}(y-1)^{|A_G|-r(A_G)+|A_H|-r(A_H)}
\\+
\sum_{A_G\subseteq E(G),A_H\subseteq E(H)} \mathbf x_{\lambda_{-}(A_G)\cup\lambda(A_H)}z^{\lambda_{r}(A_G)}(y-1)^{|A_G|-r(A_G)+|A_H|-r(A_H)}
\end{multline}
Using the previous relations we can simplify the last equation to
\begin{multline}
\sum_{A = A_G\cup A_H\cup\{e\}} \mathbf x_{\lambda_{-}(A)} z^{\lambda_{r}(A)}(y-1)^{|A|-r(A)}
\\+
\sum_{A = A_G\cup A_H} \mathbf x_{\lambda_{-}(A)}z^{\lambda_{r}(A)}(y-1)^{|A|-r(A)},
\end{multline}
where in both sums $A_G$ ranges over all subsets of $E(G)$ and $A_H$ ranges over all subsets of $E(H)$. Finally, we can combine the sums to get $U^r(G\cdot H)$, which finishes the proof. 
\end{proof}
\begin{remark}
\label{pawlowski}
It is well-known (see \cite{stanley95symmetric,noble99weighted}) that the chromatic symmetric function of a graph can be recovered from the $U$-polynomial by 
\[X(G) = (-1)^{|V(G)|}U(G;x_i=-p_i,y=0).\]
In \cite{pawlowski2018chromatic}, Pawlowski introduced the rooted chromatic symmetric function. It is not difficult to check that 
\[X^r(G,v_0) = (-1)^{|V(G)|}\frac{1}{z}U^r(G,v_0;x_i=-p_i,y=0).\]
By performing this substitution on \eqref{eq:pseudo} we obtain Proposition 3.4 in \cite{pawlowski2018chromatic}.
\end{remark}




\section{The rooted $U$-polynomial distinguishes rooted trees}
\label{sec:dist}
In this section we will show that the rooted $U$-polynomial distinguishes rooted trees up to isomorphism. Similar results for other invariants of rooted trees appear in \cite{bollobas2000polychromatic,gordon1989greedoid,hasebe2017order}. The proof given here follows closely the one in \cite{gordon1989greedoid} but one can also adapt the proof of \cite{bollobas2000polychromatic}. Before stating the result we need the two following lemmas. 

\begin{lemma}
\label{lem:degree}
Let $(T,v)$ be a rooted tree. Then, the number of vertices of $T$ and the degree of $v$ can be recognized from $U^r(T)$.
\end{lemma}
\begin{proof}
It is easy to see that {$U^r(T)=z^{n}+q(z)$} where $q(z)$ is a polynomial in $z$ of degree less than $n$ with coefficients in $\ZZ[y,\mathbf x]$ and  $n$ is the number of vertices of $T$. Hence, to recognize the number of vertices of $T$, it suffices to take the term of the form $z^j$ with the largest exponent in $U^r(T)$ and this exponent is the number of vertices. To recognize the degree of $v$, observe that a term of $U^r(T)$ has the form $zx_\lambda$ for some $\lambda$ corresponding to $A$ if and only the edges of $A$ are not incident with $v$. In particular, the term of this form with smaller degree correspond to $A=E\setminus I(v)$ where $I(v)$ denotes the set of edges that are incident with $v$ and in fact the term is $zx_{n_1}x_{n_2}\ldots x_{n_d}$ where $n_1,n_2,\ldots,n_d$ are the number of vertices in each connected component of $T-v$. Since each connected component is connected to $v$ by an edge,  this means that the degree of $v$ is equal to $d$ and hence it is the degree of this term minus one.
\end{proof}
\begin{lemma}
\label{lem:irreducible}
Let $(T,v)$ be a rooted tree. Then, 
$\frac{1}{z}U^r(T,v)$ is irreducible if and only if  the degree of $v$ is one. 
\end{lemma}
\begin{proof}
Let $n$ denote the number of vertices of $T$. Suppose that the degree of $v$ is one. We will show that $\frac{1}{z}U^r(T,v)$
is irreducible. Denote by $e$ the only edge of $T$ that is incident with $v$. It is easy to check that $\lambda_r(A)\geq 1$ for all $A\subseteq E$ and that, if $A=E-e$, then $\lambda(A)=(n-1,1)$. Consequently,  
\[\frac{1}{z}U^r(T,v) = x_{n-1}+ 
\sum_{A\subseteq E, A\neq E-e}\mathbf x_{\lambda_{-}(A)}z^{\lambda_{r}(A)-1} \]
where the second sum is a polynomial in $\ZZ[z,x_1,x_2,\ldots,x_{n-2}]$. This implies that $\frac{1}{z}U^r(T,v)$ is a monic polynomial in $x_{n-1}$ of degree one, and hence it is irreducible. 
To see the converse, it suffices to observe that if the degree of $v$ is equal to $l>1$ then there are $T_1,T_2,\ldots, T_l$ rooted trees having a root of degree one and 
$(T,v) = T_1\odot T_2\odot\cdots\odot T_l$. This implies that 
\[\frac{1}{z}U^r(T) = \frac{1}{z}U^r(T_1)\frac{1}{z}U^r(T_2)\ldots\frac{1}{z}U^r(T_l)\]
and hence $\frac{1}{z}U^r(T)$ is not irreducible. 
\end{proof}
We say that a rooted tree $(T,v)$ can be reconstructed from its $U^r$-polynomial if we can determine $(T,v)$ up to rooted isomorphism  from $U^r(T,v)$. We show the following result.  
\begin{theorem}
\label{teo8}
All rooted trees can be reconstructed from its $U^r$-polynomial. 
\end{theorem}
\begin{proof}
By  Lemma \ref{lem:degree} we can recognize the number of vertices of a rooted tree from its $U^r$-polynomial. Thus, we proceed by induction on the number of vertices. For the base case, there is only one tree with $1$ vertex, hence the assertion is trivially true. Now suppose that all rooted trees with $k-1$ vertices can be reconstructed from their $U^r$-polynomial and let $U^r(T,v)$ be the $U^r$-polynomial of some unknown tree $(T,v)$ with $k$ vertices. Again by Lemma \ref{lem:degree} we can determine the degree $d$ of $v$ from $U^r(T)$. We distinguish two cases: 

\begin{itemize}
\item $d=1$: In this case, let $T'=T-v$ rooted the unique vertex of $T$ that is incident to $v$. This means that $T = 1\cdot T'$ where $1$ is the rooted tree with only one vertex. From \eqref{eq:pseudo} it follows 
that \[U^r(T)  = z(U^r(T') + U(T'))=(\frac{1}{z}U^r(T'))z^2+U(T')z.\]
Since the variable $z$ does not appear in $U(T')$, we can determine $U^r(T')$ 
from $U^r(T)$ by collecting all the terms in the expansion of $U^r(T)$ that are divisible by $z^2$ and then dividing them by $z$. Since $T'$ has $k-1$ vertices, by the induction hypothesis, we can reconstruct $T'$ and hence the equality $T=1\cdot T'$ allows us to reconstruct $T$.
\item $d>1$: In this case, we know that $\frac{1}{z}U^r(T)$ is not irreducible by Lemma \ref{lem:irreducible} and hence it decomposes as
\[\frac{1}{z}U^r(T) = P_1P_2\cdots P_d,\]
where the $P_i$ are the irreducible factors in $\ZZ[z,x_1,\ldots]$. On the other hand, as in the proof of Lemma \ref{lem:irreducible}, $T$ can be decomposed into $d$ branches $T_1,T_2,\ldots, T_d$, which are 
rooted trees with the root having degree one, $T = T_1\cdot T_2\cdot T_d$ and 
\[\frac{1}{z}U^r(T) =\frac{1}{z}U^r(T_1)\frac{1}{z}U^r(T_2)\cdots \frac{1}{z}U^r(T_d).\]
Since $\ZZ[z,x_1,\ldots]$ is a unique factorization domain, up to reordering factors, we have $U^r(T_i) = zP_i$ for all $i\in\{1,\ldots,d\}$. Since $d>1$ and by the definition of the $T_i$'s they have at least one edge (and hence two vertices), it follows that each $T_i$ has at most $k-1$ vertices. Since we know each of their $U^r$-polynomials, by the hypothesis induction, we can reconstruct each of them, and so we can reconstruct $T$.
\end{itemize}
\end{proof}
\begin{corollary}
The $U^r$-polynomial distinguishes trees up to rooted isomorphism. 
\end{corollary}
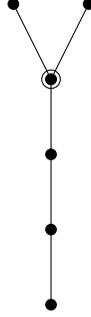
\begin{figure}
\begin{center}
\begin{tikzpicture}
\draw (0.5,0) node[fill=none,minimum width=7pt] {};
\draw (0,1) node {} -- (0.5,0)
            node {} -- (1,1)
            node {};
\draw (0.5,0) node {} -- (0.5,-1) node {}
              -- (0.5,-2) node {}
              -- (0.5,-3) node {};
\end{tikzpicture}
\end{center}
\caption{The reconstructed tree from Example \ref{example}.}
\end{figure}

\begin{example}
\label{example}
Suppose $U^r(T,v)=x_{1}^{5} z + 3 \, x_{1}^{4} z^{2} + 4 \, x_{1}^{3} z^{3} + 4 \, x_{1}^{2} z^{4} + 3 \, x_{1} z^{5} + z^{6} + 2 \, x_{1}^{3} x_{2} z + 5 \, x_{1}^{2} x_{2} z^{2} + 4 \, x_{1} x_{2} z^{3} + x_{2} z^{4} + x_{1}^{2} x_{3} z + 2 \, x_{1} x_{3} z^{2} + x_{3} z^{3}$. From the term $z^6$, we know that $T$ has $6$ vertices. The terms of the form $z\mathbf{x}_\lambda$ are $x_1^5z+2x_1^3x_2z+x_1^2x_3z$. Thus, the degree of $v$ is $3$. Moreover, if we factorize $\frac{1}{z}U^r(T,v)$ into irreducible factors we obtain
\[\frac{1}{z}U^r(T,v)={\left(x_{1}^{3} + x_{1}^{2} z + x_{1} z^{2} + z^{3} + 2 \, x_{1} x_{2} + x_{2} z + x_{3}\right)} {\left(x_{1} + z\right)}{\left(x_{1} +  z\right)}.\]
This means that 
\begin{eqnarray}
U_r(T_1,v_1)&=& x_{1}^{3}z + x_{1}^{2} z^2 + x_{1} z^{3} + z^{4} + 2 \, x_{1} x_{2} z + x_{2} z^2 + x_{3}z,\\
U_r(T_2,v_2)&=&x_{1}z + z^2,\\
U_r(T_3,v_3)&=&{x_{1}z +  z^2}.
\end{eqnarray}
From the terms $z^4$ and $x_3z$ in $U^r(T_1)$ it is easy to see that $T_1$ has $4$ vertices and $v_1$ has degree 1. Hence, $T_1=1\cdot T_1'$,
where 
\[U^r(T_1') = \frac{1}{z}\left(x_{1}^{2} z^2 + x_{2} z^2 +  x_{1} z^{3} + z^{4}\right) = x_{1}^{2} z + x_{2} z +  x_{1} z^{2} + z^{3}. \] 
Similarly $T_1'=1\cdot T_1''$, where 
\[U^r(T_1'') = \frac{1}{z}\left( x_{1} z^{2} + z^{3}\right)=x_{1} z + z^{2}.\] 
From this, it is not difficult to see that $T_2,T_3$ and $T_1''$ are rooted isomorphic to $1\cdot 1$. Finally, we have 
\[T = (1\cdot (1\cdot (1\cdot 1)))\odot (1\cdot 1)\odot (1\cdot 1).\]

\end{example}
\section{The restricted $U$-polynomial}
\label{sec:main}
Let $T$ be a tree with $n$ vertices. 
It is well known that in this case $r(A)=|A|$ for every $A\subseteq E(T)$. Hence, $U(T)$ and $U^r(T)$ (if $T$ is rooted) do not depend on $y$. Given an integer $k$, the  $U_k$-polynomial of $T$ is defined by 
\begin{equation}
U_k(T;\mathbf x)=\sum_{A\subseteq E,|A|\geq n-k}\mathbf x_{\lambda(A)}.
\end{equation}
Observe that since $T$ is a tree, every term in $U_k(T)$ has degree at most $k+1$ and that restricting the terms in the expansion of $U(T)$ to those of degree at most $k+1$ yields $U_k(T)$. 
As noted in the introduction, it is proved in \cite{Aliste2017PTE}  that for every integer $k$ there are non-isomorphic trees $T$ and $T'$ that have the same $U_k$-polynomial but distinct $U_{k+1}$-polynomial. However, the trees found in \cite{Aliste2017PTE} are not explicit. In this section, with the help of the tools developed in previous sections, we will explicitly construct such trees.  

We start by defining two sequences of rooted trees. Let us denote the path on three vertices, rooted at the central vertex, by $A_0$ and the path on three vertices, rooted at one of the leaves, by $B_0$. The trees $A_k$ and $B_k$ for $k\in\NN$ are defined inductively as follows:
\begin{equation}
\label{AK}
A_k := A_{k-1}\cdot B_{k-1}\quad\text{and}\quad
B_k := B_{k-1}\cdot A_{k-1}.
\end{equation}
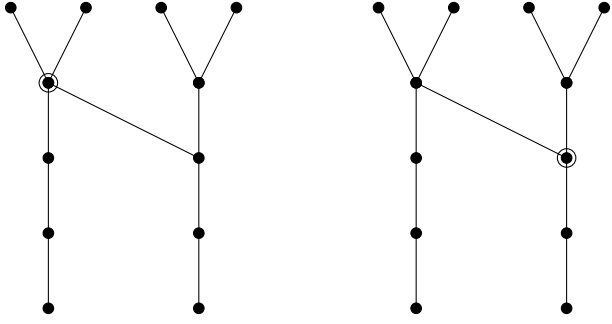
\begin{figure}
\centerline{
\begin{tikzpicture}
\draw (0.5,0) node[fill=none,minimum width=7pt] {};
\draw (0,1) node {} -- (0.5,0)
            node {} -- (1,1)
            node {};
\draw (0.5,0) node {} -- (0.5,-1) node {}
              -- (0.5,-2) node {}
              -- (0.5,-3) node {};
\draw (2,1) node {} -- (2.5,0)
           node {} -- (3,1)
            node {};
\draw (2.5,0) node {} -- (2.5,-1) node {}
              -- (2.5,-2) node {}
              -- (2.5,-3) node {};
\draw (2.5,-1) -- (0.5,0);
\end{tikzpicture}
\hspace{1.5cm}
\begin{tikzpicture}
\draw (2.5,-1) node[fill=none,minimum width=7pt] {};
\draw (0,1) node {} -- (0.5,0)
            node {} -- (1,1)
            node {};
\draw (0.5,0) node {} -- (0.5,-1) node {}
              -- (0.5,-2) node {}
              -- (0.5,-3) node {};
\draw (2,1) node {} -- (2.5,0)
           node {} -- (3,1)
            node {};
\draw (2.5,0) node {} -- (2.5,-1) node {}
              -- (2.5,-2) node {}
              -- (2.5,-3) node {};
\draw (2.5,-1) -- (0.5,0);
\end{tikzpicture}
}
\caption{The rooted trees $A_2$ and $B_2$}
\end{figure}
We first observe that $A_0$ and $B_0$ are isomorphic as unrooted trees but not isomorphic as rooted trees, which means that they have different $U^r$. In fact, a direct calculation shows that 
\[\Delta_0:=U^r(A_0)-U^r(B_0) = x_1z^2-x_2z.\]
By applying Lemma \ref{lemma:joining} we deduce:
\begin{proposition}
\label{prop}
For all $k\in\NN$, the trees $A_k$ and $B_k$ are isomorphic but not rooted-isomorphic.  Moreover, we have  
\begin{equation}
\label{DeltaK}
U^r(A_k) - U^r(B_k) = \Delta_0P_k,
\end{equation} 
where $P_k := U(A_0)U(A_1)\cdots U(A_{k-1})$.
\end{proposition}
\begin{proof}
The proof is done by induction. The basis step is clear from the definition of $\Delta_0$. For the induction step, 
we assume that for a given $k$, the graphs $A_{k-1}$ and $B_{k-1}$ are isomorphic and that $U^r(A_{k-1}) - U^r(B_{k-1}) = \Delta_0P_{k-1}$.
From \eqref{AK}, it is easy to see that $A_k$ and $B_k$ are isomorphic as unrooted trees. Also, combining \eqref{AK} with \eqref{eq:sep_concat} we get 
\[U^r(A_k) = U^r(A_{k-1})(U^r(B_{k-1})+U(B_{k-1})).\]
Similarly for $B_k$ we get
\[U^r(B_k) = U^r(B_{k-1})(U^r(A_{k-1})+U(A_{k-1})).\]
Subtracting these two equations, using that $U(A_{k-1})=U(B_{k-1})$ and plugging the induction hypothesis yields 
\[U^r(A_k) - U^r(B_k) = U(A_{k-1})\big(U^r(A_{k-1})-U^r(B_{k-1})\big) = U(A_{k-1})P_{k-1}\Delta_0 = P_{k}\Delta_0.\]
Hence, by induction, \eqref{DeltaK} holds for every $k$. To finish the proof, notice that since $A_k$ and $B_k$ have distinct $U^r$, they are not rooted-isomorphic by Theorem~\ref{teo8}.
\end{proof}
Observe that all the terms of $P_k$ have degree at least $k$. Now we can state our main result.
\begin{theorem}\label{theo:YZ}
Given $k,l\in\NN$, let
\begin{equation}
Y_{k,l}={(A_k\odot A_l)\cdot (B_k\odot B_l)}\quad \text{and}\quad
Z_{k,l} = (A_l \odot B_k)\cdot (B_l\odot A_k).
\end{equation}
Then the graphs $Y_{k,l}$ and $Z_{k,l}$ (seen as unrooted trees) are not isomorphic, have the same $U_{k+l+2}$-polynomial and distinct $U_{k+l+3}$-polynomial.
\end{theorem}
Before giving the proof, we need the following lemma, which is 
a corollary of Lemma \ref{lemma:joining} and Proposition \ref{prop}.

\begin{lemma}\label{l:D}
Let $T$ be a rooted tree and $i$ an integer. Then 
\begin{equation}
\label{eq:lD}
U(A_i\odot T) - U(B_i\odot T) =   P_i\mathcal{D}(T),
\end{equation}
where \begin{equation}
\label{eq:DT}
\mathcal{D}(T) = x_1(z U^r(T))^* - x_2 U(T).
\end{equation} In particular all the terms in $\mathcal{D}(T)$ have degree at least $2$.
\end{lemma}

\begin{proof}
By Lemma \ref{lemma:joining}, we have
\[
U^r(A_i \odot T)- U^r(B_i \odot T) = z^{-1}U^r(T) \big(U^r(A_i)-U^r(B_i)\big).
\]
Applying Proposition \ref{prop} to the last term yields
\[U^r(A_i \odot T)- U^r(B_i \odot T) = P_iU^r(T)\frac{\Delta_0}{z}.
\]
The conclusion now follows by taking the specialization $z^n\rightarrow x_n$ in the last equation to obtain (note that $P_i$ does not depend on $z$)
\[U(A_i\odot T) - U(B_i\odot T) = P_i \left[U^r(T)(x_1z-x_2)\right]^* = P_i\mathcal{D}(T).\]
\end{proof}
\begin{proof}[Proof of Theorem \ref{theo:YZ}]
We start by applying the deletion-contraction formula to the edges corresponding to the $\cdot$  operation in the definitions of $Y_{k,l}$ and $Z_{k,l}$; it is easy to see that 
\begin{equation} \label{eq:first_difference}
U(Y_{k,l}) - U(Z_{k,l}) = U(A_k \odot A_l) U(B_k\odot B_l) - U(A_l\odot B_k)U(B_l\odot A_k),
\end{equation} 
since after contracting the respective edges we get isomorphic weighted trees.

We apply Lemma~\ref{l:D} twice, to $T = A_k$ and $i=l$ first, and then to $T=B_k$ and $i=l$, and replace the terms $U(A_k\odot A_l)$ and $U(A_l\odot B_k)$ in~\eqref{eq:first_difference}. Recalling that $\odot$ is commutative and after some cancellations, we obtain
$$U(Y_{k,l})-U(Z_{k,l}) =  P_l\Big( \mathcal{D}(A_k) U(B_k\odot B_l) -\mathcal{D}(B_k) U(B_l\odot A_k) \Big).$$
We use Lemma~\ref{l:D} once more, with $T=B_l$ and $i=k$, to arrive at
\begin{equation}
\label{eq:ykzk:final}
U(Y_{k,l})-U(Z_{k,l}) =  P_l \Big(\big(\mathcal{D}(A_k)-\mathcal{D}(B_k)\big) U(B_l\odot A_k) - \mathcal{D}(A_k)\mathcal{D}(B_l)P_k\Big)
\end{equation}
Using \eqref{eq:DT} and Proposition~\ref{prop} we get
\[\mathcal{D}(A_k)-\mathcal{D}(B_k) = x_1P_k(z\Delta_0)^* = x_1(x_1x_3-x_2^2)P_k,\]
 and substituting this into \eqref{eq:ykzk:final} yields
\begin{equation}
\label{eq:last}
U(Y_{k,l})-U(Z_{k,l}) = P_lP_k\Big((x_1^2x_3-x_1x_2^2)U(B_l\odot A_k) - \mathcal{D}(A_k)\mathcal{D}(B_l)\Big).
\end{equation}
This implies that all the terms that appear in the difference have degree at least $l+ k + 4$. Hence $Y_{k,l}$ and $Z_{k,l}$ have the same $U_{k+l+2}$-polynomial. To see that they have distinct $U_{k+l+3}$-polynomial, from \eqref{eq:last} we can deduce that the only terms of degree $l+k+4$ come from terms of degree $4$ in the difference 
\[
\Big((x_1^2x_3-x_1x_2^2)U(B_l\odot A_k) - \mathcal{D}(A_k)\mathcal{D}(B_l)\Big).\]
An explicit computation of these terms yields
\[ (x_1^2x_3-x_1x_2^2)x_{n(l)+n(k)-1}-(x_1 x_{n(k)+1}-x_2x_{n(k)})(x_1x_{n(l)+1}-x_2x_{n(l)}),\]
where $n(k)$ is the number of vertices of $A_k$ (and also $B_k$). From this last equation, the conclusion follows. 
\end{proof}

We may consider the following quantity: 
\[\Phi(m) := \min\{l: \exists \text{ non-isomorphic trees $H,G$ of size $l$ s.t. $U_m(H)=U_m(G)$}\}.\]
\begin{proposition}
\label{prop:Phi}
We have 
\[\Phi(m)\leq 
\begin{cases} 
6\cdot 2^{\frac{m}{2}}-2, &\text{if $m$ is even}\\
6\cdot 3\cdot 2^{\lfloor\frac{m}{2}\rfloor-1}-2,& \text{if $m$ is odd.}
\end{cases}\]
In particular $\Phi(m)$ is finite.
\end{proposition}
\begin{proof}
By Theorem \ref{theo:YZ}, we see that $\Phi(m)\leq|Y_{k,l}|$ for all $(k,l)$ such that $k+l+2=m$. It is easy to check that $|A_i|=|B_i|=3\cdot 2^i$ for all $i$. Thus, 
\[|Y_{k,l}|= 2(|A_k|+|B_l|-1)=6(2^k+2^l)-2\quad\text{for all $(k,l)$}.\]
 If $m=k+l+2$ is fixed, then we see that $|Y_{k,l}|$ is minimized
when $k=l=\frac{m}{2}-1$ if  $m$ is even and otherwise is minimized when $k=\lfloor\frac{m}{2}\rfloor$ and $l=\lfloor\frac{m}{2}\rfloor-1$. Replacing the values of $k$ and $l$ yields the desired inequality.
\end{proof}
Observe that when $(k,l) \in \{ (0,0), (1,0), (1,1)\}$ (respectively), the graphs $Y_{k,l}$ and $Z_{k,l}$ are the smallest examples of non-isomorphic trees with the same $U_m$ for $m\in \{2,3,4\}$ (respectively). This fact was verified computationally in \cite{smith2015symmetric}. This leads us to make the following conjecture 
\begin{conjecture}\label{c:YZ}
If $m$ is even, then $Y_{m/2-1,m/2-1}$ and $Z_{m/2-1,m/2-1}$ are the smallest non-isomorphic trees with the same $U_m$-polynomial and if $m$ is odd, then 
the same is true for $Y_{\lfloor m/2\rfloor,\lfloor m/2\rfloor-1}$ and $Z_{\lfloor m/2\rfloor,\lfloor m/2\rfloor-1}$. In other words, 
\[\Phi(m) = 
\begin{cases} 
6\cdot 2^{\frac{m}{2}}-2, &\text{if $m$ is even}\\
6\cdot 3\cdot 2^{\lfloor\frac{m}{2}\rfloor-1}-2,& \text{if $m$ is odd.}
\end{cases}\]
\end{conjecture}
The following proposition relates $\Phi$ with Stanley's conjecture. 
\begin{proposition}
The following assertions are true:
\begin{enumerate}[a)]
\item For every $m$, Stanley's conjecture is true for trees with at most $\Phi(m)-1$ vertices.
\item Stanley's conjecture is true if and only if $\lim_m\Phi(m)=\infty$. 
\item Conjecture \ref{c:YZ} implies Stanley's conjecture.
\end{enumerate} 
\end{proposition}
\begin{proof}
To show a), observe that the existence of non-isomorphic trees $T$ and $T'$ of size smaller than $\Phi(m)$ with the same $U$-polynomial contradicts the definition of $\Phi(m)$. To see b), if $\lim_m\Phi(m)=\infty$, then by a), then clearly Stanley's conjecture is true for all (finite) trees. For the converse, suppose that $\Phi(m)$ is uniformly bounded by $N$, and let $T_m,T_m'$ be two non-isomorphic trees of size smaller or equal than $N$ with the same $U_m$-polynomial. Since there finitely many pairs of trees of size smaller or equal than $N$, it follows that there exist $T$ and $T'$ two trees such that $T=T_m$ and $T'=T'_m$ for infinitely many $m$. This implies that $U(T)=U(T')$ and this would contradict Stanley's conjecture. This finish the proof of b). Assertion c) follows directly from Conjecture \ref{c:YZ} and b). 
\end{proof}


\section*{Acknowledgments}
The first and third author are partially supported by CONICYT FONDECYT Regular 1160975 and Basal PFB-03 CMM Universidad de Chile. The second author is partially supported by the Spanish
Ministerio de Economía y Competitividad project MTM2017-82166-P. A short version of this work appeared in \cite{aliste2018dmd}.




\end{document}